\documentclass[12pt,reqno]{amsart}

\newcommand\version{May 13, 2013}

\usepackage{amsmath,amsfonts,amsthm,amssymb,amsxtra}

\newtheorem{theorem}{Theorem}[section]

\newtheorem{lemma}[theorem]{Lemma}

\theoremstyle{definition}

\theoremstyle{remark}

\newtheorem{remark}[theorem]{Remark}

\numberwithin{equation}{section}

\newcommand{\const}{\mathrm{const}\ }

\renewcommand{\epsilon}{\varepsilon}

\newcommand{\R}{\mathbb{R}}

\DeclareMathOperator{\spa}{span}

\begin{document}

\title[Stablity estimates for eigenvalues --- \version]{Stability estimates for the lowest eigenvalue\\ of a Schr\"odinger operator}

\author{Eric A. Carlen}
\address{Eric A. Carlen, Department of Mathematics, Hill Center,
Rutgers University, 110 Frelinghuysen Road, Piscataway NJ 08854-8019, USA}
\email{carlen@math.rutgers.edu}

\author{Rupert L. Frank}
\address{Rupert L. Frank, Department of Mathematics,
Princeton University, Princeton, NJ 08544, USA}
\email{rlfrank@math.princeton.edu}

\author{Elliott H. Lieb}
\address{Elliott H. Lieb, Departments of Mathematics and Physics,
Princeton University, Princeton, NJ 08544,
USA}
\email{lieb@princeton.edu}

\thanks{\copyright\, 2012 by the authors. This paper may be reproduced, in
its entirety, for non-commercial purposes.\\
Work partially supported by NSF grants DMS--0901632 (E.A.C.), PHY--1068285
(R.L.F.), PHY--0965859 (E.H.L.) and the Simons Foundation grant \#230207
(E.H.L.)}

\begin{abstract}
There is a family of potentials that minimize the lowest eigenvalue of a Schr\"odinger eigenvalue under the constraint of a given $L^p$ norm of the potential. We give effective estimates for the amount by which the eigenvalue increases when the potential is not one of these optimal potentials. Our results are analogous to those for the isoperimetric problem and the Sobolev inequality. We also prove a stability estimate for H\"older's inequality, which we believe to be new.
\end{abstract}

\maketitle
\centerline{\version}
%%%%%%%%%%%%%%%%%%%%%%%%%%%%%%%%%%%%%%%%%%%%%%%%%%%%%%%%%%%%%%%%%%%%%%%%%%%%%%%%%%%%%%%%%%%

\section{Introduction}

Recently there has been considerable interest and considerable progress on
the question of stability estimates for various inequalities, both geometric
and analytic. This means finding lower bounds for deviations in energies
from their minimum values in terms of the distance to the nearest energy
minimizing configuration.
Our goal here is to extend these ideas to the lowest eigenvalue of the
Schr\"odinger operator $-\Delta+V$ in $L^2(\R^d)$.  This lowest eigenvalue
is the ground state energy of a quantum mechanical particle moving in the
potential $V$.

In 1961, Keller \cite{Ke} raised and answered the question of
finding the potential that minimizes the lowest eigenvalue
$\lambda(V)$ of $-\Delta+V$ in $d=1$ under the condition of a
specified $L^p$ norm of $V$. In other words, Keller was interested in the
sharp constant $\mathcal C_{\gamma,d}$ in the inequality
$$
\lambda(V) \geq -\mathcal C_{\gamma,d} \left( \int_{\R^d}
V_-^{\gamma+d/2} \,dx \right)^{1/\gamma} \, , 
$$
where $V_-$ is the negative part of $V$. (Our inequalities are stated
without reference to $V_+$ because $\lambda (V) \geq \lambda (-V_-)$ and
may be arbitrarily close to $ \lambda (-V_-)$ 
even if $V_+ $ is large, but supported far from the support of $V_-$.) 
Keller solved the $d=1$ problem and computed $\mathcal C_{\gamma,1}$ by
explicitly computing the optimal potential, which cannot be
explicitly computed  for
$d\geq 2$.
The value of $\mathcal C_{\gamma,d}$ for $d=2,3$ was found numerically in
\cite{LT} in the more general context of the entire negative eigenvalue part
of the spectrum. One of
the things we do in this paper is to prove the existence and uniqueness (up
to translations and dilations) of an optimum $V$. Our uniqueness proof
relies on a celebrated result of Kwong \cite{K} who extended and completed
work of Coffman \cite{C} and McLeod and Serrin \cite{MS}.

We deal with the {\it stability question}: If  for some $r>0$, the distance between
$V$ and each optimizing potential $W$ is at least $r$, by how much must 
$\lambda(V) \left( \int_{\R^d} V_-^{\gamma+d/2} \,dx \right)^{-1/\gamma}$
deviate from  the minimal value
$-\mathcal{C}_{\gamma,d}$?  Our main result is a lower bound on this energy deviation. To describe it briefly, let us assume for the moment that $V$
is non-positive. 
The way we measure distance depends on the value of $p := \gamma+d/2$. For 
$p\geq 4$,  our lower bound on the energy deviation is a multiple of  infimum of $\|V - W\|_p^2$ as $W$ 
ranges over the set  of minimizing potentials.  For 
$2\leq p\leq 4$, our lower bound is a multiple of  $\|(-V)^{p/p'} - (-W)^{p/p'}\|_{p'}^2$
as $W$ 
ranges over the set  of minimizing potentials, and where 
where $p'$ is the dual index to $p$.  In either case, our lower bound
is quadratic, as one might hope. 

Our analysis requires the concatenation of two inequalities. One is a
stability estimate for the Gagliardo-Nirenberg-Sobolev (GNS) inequality,
proved in Section~\ref{sec:gns}.  The second is a stability estimate for
H\"older's inequality $|\int fg| \leq \Vert f \Vert_p \Vert g \Vert_{p'}$
which is well known to be an equality if and only if $|f|^{p'} $ and $|g|^p$
are proportional and $fg$ has constant phase. Our Theorem~\ref{holder}
says that when there is near inequality these conditions are nearly
satisfied.

This problem belongs to a larger class of problems of which 
a classic example, going back to Bonnesen \cite{Bo} in 1924, is to
quantitatively estimate  how much the ratio of the square of the
perimeter to the area of a planar domain $\Omega$
exceeds that of a disc --- the isoperimetric minimizer --- in terms of
some measure of the deviation of $\Omega$ from a disc. 
Inequalities of this kind were
termed Bonnesen-type inequalities by Osserman \cite{O}. 

Hall, Hayman and Weitsman \cite{HHW} and Hall \cite{Ha} provided the first
stability theorem for the isoperimetric inequality in  $d \geq
3$. 
They introduced the {\it Fraenkel asymmetry} as an approporate measure of
deviation from a ball and proved that the isoperimetric deficit was at least
quartic in this asymmetry. The proof that the deficit is actually 
quadratic in the 
assymmetry was finally achieved 
by Fusco, Maggi and Pratelli
\cite{FMP1}. See the review paper \cite{Ma}. 

Other problems of this type have also been addressed recently. The
classic Faber--Krahn inequality says that the lowest eigenvalue of the
Dirichlet Laplacian in a domain of unit volume occurs for the ball. The
increase of the lowest eigenvalue when the domain is not a ball has been
quantified in another paper of Fusco, Maggi and Pratelli \cite{FMP2} and, in
fact, they did it for the $p$-Laplacian, not just for $p=2$.

A closely related question is what happens to the $p$-Laplacian Sobolev quotient for functions that deviate from the famous ratio-minimizing functions $(1+|x|^{p'})^{-(d-p)/p}$, found by Bliss \cite{B}, Rosen \cite{R}, Aubin \cite{A} and Talenti \cite{T}. For $p=2$, this question was raised in \cite{BL} and answered six years later in \cite{BE} by means of a clever compactness argument, which we will find useful in the present paper. Result for $p\neq 2$ were recently found by Cianchi, Fusco, Maggi and Pratelli \cite{CFMP}.

%%%%%%%%%%%%%%%%%%%%%%%%%%%%%%%%%%%%%%%%%%%%%%%%%%%%%%%%%%%%%%%%%%%%%%%%%%%%
%%%%%%%%%%%%%%%

\section{Main result and method of proof}

\subsection{Statement of the main result}

We define
\begin{equation}
 \label{eq:ev}
\lambda(V) = \inf\left\{ \int_{\R^d} \left( |\nabla\psi|^2 + V|\psi|^2 \right) \,dx :\ \psi\in H^1(\R^d)\,, \|\psi\|_2=1 \right\} \,.
\end{equation}
Here $\|\cdot\|_q$ denotes the $L^q$-norm. It is well-known \cite[Ch.
11]{LiLo} that $\lambda(V)$ is finite if, for instance, $V\in
L^p(\R^d)$ for $p\geq 1$ in $d=1$, for $p>1$ in $d=2$ and
$p\geq d/2$ if $d\geq 3$. Moreover, if $V_-$ vanishes at infinity in the
sense that $|\{ V_-> \tau\}|<\infty$ for any $\tau>0$, then $\lambda(V)\leq
0$. In case $\lambda(V)$ is finite and strict inequality $\lambda(V)< 0$
holds, an optimizing $\psi$ exists and $\lambda(V)$ is the smallest
eigenvalue of the operator $-\Delta +V$ in $L^2(\R^d)$.

Clearly, in order to define our problem we need to know the following theorem, which will be proved in the following subsection and in Section \ref{sec:gns0}.

\begin{theorem}[Existence and uniqueness of an optimal potential]\label{unique}
 Let $\gamma> 1/2$ if $d=1$ and $\gamma>0$ if $d\geq 2$ and define
$$
\mathcal C_{\gamma,d} := \sup \frac{|\lambda(V)|}{\left( \int_{\R^d} V_-^{\gamma+d/2} \,dx\right)^{1/\gamma} } \,.
$$
Then there is a non-positive function $\mathcal V$, which is unique up to translations and dilatons, such that
\begin{align}\label{eq:m}
 \mathcal M & = \left\{ V \in L^{\gamma+d/2}(\R^d) :\ |\lambda(V)| = \mathcal C_{\gamma,d} \left( \int_{\R^d} V_-^{\gamma+d/2} \,dx \right)^{1/\gamma} \right\} \nonumber \\
& = \left\{ b^2 \mathcal V(b(\cdot-a)) : \ b>0\,, a\in\R^d \right\} \,.
\end{align}
\end{theorem}

Our problem and our main result is the following stability bound for the optimization problem related to $\mathcal C_{\gamma,d}$, which will be proved in Subsection \ref{sec:main}, with the aid of some results that will be proved in Sections \ref{sec:holder} and \ref{sec:gns}.

It is useful to define a parameter $q$ so that $L^{q/2}$ is the dual of $L^{\gamma+d/2}$. That is:
\begin{equation}
 \label{eq:dual}
\frac{1}{\gamma+d/2} + \frac 2q = 1 \,, \qquad \mathrm{i.e.,} \qquad 
\gamma+d/2 = \frac{q}{q-2} \,.
\end{equation}

\begin{theorem}[Stability]\label{main}
 Let $\gamma> 1/2$ if $d=1$ and $\gamma>0$ if $d\geq 2$. Then:
\begin{enumerate}
\item[(i)] For $\gamma + d/2 \leq 2$, 
 there is a constant $c_{\gamma,d}>0$ such that for any $V\in L^{\gamma+d/2}(\R^d)$,
\begin{equation} \label{eq:main1}
\frac{|\lambda(V)|}{\left(\int_{\R^d} V_-^{\gamma+d/2} \,dx\right)^{1/\gamma}} \leq \mathcal C_{\gamma,d} \left( 1 - c_{\gamma,d} \inf_{W\in\mathcal M} \frac{\|V_--W_-\|_{\gamma+d/2}^2}{\|V_-\|_{\gamma+d/2}^2} \right) \,.
\end{equation}
\item[(ii)] For  $\gamma + d/2 \geq 2$,   there is a constant $c_{\gamma,d}>0$ such that for any $V\in L^{\gamma+d/2}(\R^d)$,
\begin{equation} \label{eq:main2}
\frac{|\lambda(V)|}{\left(\int_{\R^d} V_-^{\gamma+d/2} \,dx\right)^{1/\gamma}} \leq \mathcal C_{\gamma,d} 
\left( 1 - c_{\gamma,d} \inf_{W\in\mathcal M} 
\frac{ \left\|V_-^{2/(q-2)}-W_-^{2/(q-2)}\right\|^2_{q/2}}{\left\|V_-^{2/(q-2)}\right\|_{q/2}^2  } \right) \,,
\end{equation}
where $q$ is related to $\gamma$ and $d$ by (\ref{eq:dual}).
\end{enumerate}
\end{theorem}

\begin{remark}
 Indeed, we prove slightly stronger bounds: The infima in \eqref{eq:main1} and \eqref{eq:main2} can be restricted to $W\in\mathcal M$ such that $\|W_-\|_{\gamma+d/2}=\|V_-\|_{\gamma+d/2}$ and $\|W_-^{2/(q-2)}\|_{q/2}=\|V_-^{2/(q-2)}\|_{q/2}$, respectively. This fixes the scale parameter $b$ in \eqref{eq:m}, and the only variation is over translations $a$.
\end{remark}

\begin{remark} Suppose $V$ and $W$ satisfy $\|V_-\|_{\gamma+d/2} =  \|W_-\|_{\gamma+d/2}$. As a consequence of \eqref{ng1} below, we have that
\begin{equation}\label{trans}
\left(\gamma+\frac{d}{2}\right)^{-1} \left(\frac12 \|V_-- W_-\|_{\gamma+d/2}\right)^{\gamma+d/2-1} \leq \left\|V_-^{2/(q-2)} - W_-^{2/(q-2)}\right\|_{q/2}\ .
\end{equation}
Hence it follows from this, the previous remark and part (ii) of Theorem~\ref{main} that:
\begin{enumerate}
\item[(iii)] For  $\gamma + d/2 \geq 2$, there is a constant $\tilde c_{\gamma,d}>0$ such that for any $V\in L^{\gamma+d/2}(\R^d)$,
\begin{equation} \label{eq:main3}
\frac{|\lambda(V)|}{\left(\int_{\R^d} V_-^{\gamma+d/2} \,dx\right)^{1/\gamma}} \leq \mathcal C_{\gamma,d} 
\left( 1 - \tilde c_{\gamma,d} \inf_{W\in\mathcal M} 
\frac{ \|V_--W_-\|^{2\gamma+d-2}_{\gamma+d/2}}{\|V_-\|_{\gamma+d/2}^{2\gamma+d-2}  } \right) \,.
\end{equation}
\end{enumerate}

Thus, we may always express our stability bound in terms of $\|V_--W_-\|_{\gamma+d/2}$, but for $\gamma + d/2 \geq 2$,
our bound will not be quadratic in this norm. It is, however, quadratic in
the $q/2$ norm, as in \eqref{eq:main2}.
\end{remark}

%%%%%%%%%%%%%%%%%%%%%%%%%%%%%%%%%%%%%%%%%%%%%%%%%%%%%%%%%%%%%%%%%%%%%%%%%%%%%%%%%%%%%%%%%%%

\subsection{Viewpoint}

We begin with some general remarks about our proof strategy for proving both Theorem \ref{unique} and Theorem \ref{main}. 

The definition \eqref{eq:ev} implies that the optimization problem for $\mathcal C_{\gamma,d}$ can be written as a double infimum over both $V$ and $\psi$,
$$
\mathcal C_{\gamma,d} = - \inf\left\{ \frac{\int_{\R^d} \left(|\nabla\psi|^2 +V \psi^2 \right) dx}{ \left( \int_{\R^d} |V|^{\gamma+d/2} \,dx \right)^{1/\gamma} \left( \int_{\R^d} \psi^2 \,dx \right) } :\  \psi\in H^1(\R^d)\,,\ V\in L^{\gamma+d/2}(\R^d) \right\} \,.
$$
Since the quotient in this formula remains invariant if we replace both $V(x)$ by $b^2 V(bx)$ and $\psi(x)$ by $c\psi(bx)$ for arbitrary $b,c>0$, we can restrict the infimum to potentials $V$ with $\int_{\R^d} |V|^{\gamma+d/2} \,dx=1$ and to functions $\psi$ with $\int_{\R^d} \psi^2 \,dx=1$. Moreover, since the quotient does not increase if we replace $V$ by $-|V|$, we can restrict the infimum to potentials $V\leq 0$. We will use the notation $V=-U$ and summarize these findings as
$$
\mathcal C_{\gamma,d} = - \inf\left\{ \int_{\R^d} \left(|\nabla\psi|^2 -U \psi^2 \right) dx :\  U\geq 0\,,\ \int_{\R^d} U^{\gamma+d/2} \,dx = \int_{\R^d} \psi^2 \,dx =1 \right\} \,.
$$
If $q$ is related to $\gamma$ by (\ref{eq:dual}), 
then our assumptions on $\gamma$ imply that $2<q< \infty$ if $d=1,2$ and $2<q<2d/(d-2)$ if $d\geq 3$. 
With $q$ related to $\gamma$ in this way, define
\begin{equation}\label{edef}
\mathcal E[\psi] = \int_{\R^d} |\nabla\psi|^2 \,dx - \left( \int_{\R^d} |\psi|^q \,dx \right)^{2/q}
\end{equation}and
\begin{equation}\label{hdef}
\mathcal H[\psi,U] = \left( \int_{\R^d} |\psi|^q \,dx \right)^{2/q} - \
int_{\R^d} U \psi^2 \,dx \,.
\end{equation}
We then write
\begin{equation}\label{eq:constant}
\mathcal C_{\gamma,d} = - \inf\left\{ \mathcal E[\psi] + \mathcal H[\psi,U]
:\  U\geq 0\,,\ \int_{\R^d} U^{\gamma+d/2} \,dx = \int_{\R^d} \psi^2 \,dx =1
\right\} \,.
\end{equation}
Because of definition \eqref{eq:dual}, the normalization $\int_{\R^d} U^{\gamma+d/2}\,dx=1$ and H\"older's inequality, 
\begin{equation}
 \label{eq:h}
\mathcal H[\psi,U]\geq 0
\end{equation}
with equality if and only if $U=(|\psi|/\|\psi\|_q)^{q-2}$. To summarize, we have shown that
\begin{equation}
 \label{eq:inf}
-\mathcal C_{\gamma,d}  = \inf\left\{ \mathcal E[\psi] :\ \psi\in H^1(\R^d)\,,\ \int_{\R^d} \psi^2\,dx =1 \right\}
\end{equation}
and that the family $\mathcal M$ of optimal potentials, defined by
$$
\mathcal M = \left\{ V \in L^{\gamma+d/2}(\R^d) :\ |\lambda(V)| = \mathcal C_{\gamma,d} \left( \int_{\R^d} V_-^{\gamma+d/2} \,dx \right)^{1/\gamma} \right\} \,,
$$
is related to the family $\mathcal G$ of optimizers in \eqref{eq:inf}, defined by
\begin{equation} \label{eq:g}
\mathcal G = \left\{ \psi\in H^1(\R^d):\ \mathcal E[\psi] = -\mathcal C_{\gamma,d} \ \text{and}\ \int_{\R^d} \psi^2\,dx =1 \right\} \,,
\end{equation}
through the relation
\begin{equation}
 \label{eq:keller2}
\mathcal M = \left\{ - b^2 \frac{|\psi(b\, \cdot)|^{q-2}}{\|\psi\|_q^{q-2}} :\ \psi \in \mathcal G \,, b>0 \right\} \,.
\end{equation}

In Section \ref{sec:gns0} we shall recall three facts: First, the infimum on the right side of \eqref{eq:inf} is finite. Second, the infimum on the right side of \eqref{eq:inf} is attained, that is, $\mathcal G\neq \emptyset$. Third, the minimizer in \eqref{eq:inf} is unique up to translations and a sign. These three facts follow with only little effort from known results, the deepest of which is the uniqueness theorem of Kwong in dimensions $d\geq 2$ \cite{K}.

These three facts, together with \eqref{eq:inf} and \eqref{eq:keller2}, immediately imply Theorem \ref{unique}.

The discussion has also placed in view a strategy for proving Theorem~\ref{main}: Theorem~\ref{main} shall follow from
a stability analysis for H\"older's inequality \eqref{eq:h} and for the Gagliardo--Nirenberg--Sobolev (GNS)-type inequality \eqref{eq:inf}. This is carried out  in the next subsection.

%%%%%%%%%%%%%%%%%%%%%%%%%%%%%%%%%%%%%%%%

%%%%%%%%%%%%%%%%%%%%%%%%%%%%%%%%%%%%%%%%%%%%%%%%%%%%%%%%%%%%%%%%%%%%%

\subsection{Proof of Theorem \ref{main}}\label{sec:main}

As indicated earlier, the proof of Theorem \ref{main} will be accomplished using two stability theorems, one for H\"older's inequality, and another for the GNS inequality (\ref{gns}). We state both stability theorems here in a form suitable for our application and refer for more general versions to Sections \ref{sec:holder} and \ref{sec:gns}.

The  stability analysis of  H\"older's inequality yields the following theorem which is proved in Section \ref{sec:holder}.

\begin{theorem}\label{remainder}  For $q\geq 2$, let $0\leq U\in L^{q/(q-2)}$ have $\|U\|_{q/(q-2)} =1$, and let $\psi\in L^{q}$. Then
for $q\geq 4$, 
\begin{equation} \label{hlow1}
\mathcal H[\psi,U] \geq \frac{1}{2(q-2)} \|\psi\|_q^2 \left\Vert  \frac{|\psi|^{q-2}}{\|\psi\|_q^{q-2}} - U\right\Vert_{q/(q-2)}^2\ ,
\end{equation}
while for $2< q < 4$, 
\begin{equation} \label{hlow2}
\mathcal H[\psi,U] \geq \frac{q-2}{8} \|\psi\|_q^2 \left\Vert  \frac{|\psi|^{2}}{\|\psi\|_q^{2}} - U^{2/(q-2)}\right\Vert_{q/2}^2\ .
\end{equation}
\end{theorem}

The other result we require is a stability theorem for the GNS inequality. The following theorem will be proved in Section~\ref{sec:gns}.

\begin{theorem}\label{becor}
Let $2< q<\infty$ if $d=1,2$ and $2< q<2d/(d-2)$ if $d\geq 3$. Let $\psi\in H^1(\R^d)$ with $\int_{\R^d} \psi^2\,dx =1$.
\if false, and define
$$
\alpha = \begin{cases}
           1/(q-2) & \text{if}\ 2<q<3 \,, \\
1 & \text{if}\ q\geq 3 \,.
        \end{cases}
$$
\fi
 Then there is a constant $c_{q,d}>0$ such that for $q\geq 4$, 
\begin{align}\label{rem1}
&\mathcal E[\psi] \geq -\mathcal C_{\gamma,d} \nonumber\\
& \qquad + c_{q,d} \inf_{\phi\in\mathcal G} \left( \max\left\{\|\psi\|_q^2,\|\phi\|_q^2\right\} \left\| \frac{|\psi|^{q-2}}{\|\psi\|_q^{q-2}} - \frac{|\phi|^{q-2}}{\|\phi\|_q^{q-2}} \right\|_{q/(q-2)}^{2} + \left| \|\psi\|_q - \|\phi\|_q \right|^2 \right)\, ,
\end{align}
and such that for $2 < q < 4$,
\begin{align}\label{rem2}
&\mathcal E[\psi] \geq -\mathcal C_{\gamma,d} \nonumber\\
& \qquad + c_{q,d} \inf_{\phi\in\mathcal G} \left( \max\left\{\|\psi\|_q^2,\|\phi\|_q^2\right\} \left\| \frac{|\psi|^{2}}{\|\psi\|_q^{2}} - \frac{|\phi|^{2}}{\|\phi\|_q^{2}} \right\|_{q/2}^{2} + \left| \|\psi\|_q - \|\phi\|_q \right|^2 \right)\ .
\end{align}
\end{theorem}

Accepting Theorems \ref{remainder} and \ref{becor} for the moment, we now explain how they can be used to derive the stability bound in Theorem \ref{main}.

\begin{proof}[Proof of Theorem \ref{main}]
We first consider the case $2 < q < 4$.  By (\ref{rem2}) and (\ref{hlow2}),
\begin{eqnarray}\mathcal{E}[\psi] +  \mathcal H[\psi,U]  + \mathcal C_{\gamma,d} &\geq&
 c_{q,d}  \inf_{\phi\in\mathcal G}\left( \|\phi\|_q^2  \left\| \frac{|\psi|^{2}}{\|\psi\|_q^{2}} - \frac{|\phi|^{2}}{\|\phi\|_q^{2}} \right\|_{q/2}^{2}  
 + \left| \|\psi\|_q - \|\phi\|_q \right|^2\right)\nonumber\\
 &+& \frac{q-2}{8} \|\psi\|_q^2 \left\Vert  \frac{|\psi|^{2}}{\|\psi\|_q^{2}} - U^{2/(q-2)}\right\Vert_{q/2}^2\ .\nonumber
 \end{eqnarray}
 Then, since
 $$\|\psi\|_q^2 \geq \frac12 \|\phi\|_q^2 - \left(\|\psi\|_q - \|\phi\|_q\right)^2\ ,$$
 and since ${\displaystyle   \left\Vert  \frac{|\psi|^{2}}{\|\psi\|_q^{2}} - U^{2/(q-2)}\right\Vert_{q/2} \leq 2}$,
 there exists a positive constant $ \tilde c_{q,d}$ such that
 \begin{eqnarray}
 \mathcal{E}[\psi] +  \mathcal H[\psi,U]  + \mathcal C_{\gamma,d}  
&\geq& \tilde c_{q,d} \inf_{\phi\in\mathcal G}\left( \|\phi\|_q^2 \left( \left\| \frac{|\psi|^{2}}{\|\psi\|_q^{2}} - \frac{|\phi|^{2}}{\|\phi\|_q^{2}} \right\|_{q/2}^{2} 
+ \left\Vert  \frac{|\psi|^{2}}{\|\psi\|_q^{2}} - U^{2/(q-2)}\right\Vert_{q/2}^2\right) \right) \nonumber\\
&\geq& \frac{\tilde c_{q,d}}{2} \inf_{\phi\in\mathcal G}\left( \|\phi\|_q^2
\left\Vert  \frac{|\phi|^{2}}{\|\phi\|_q^{2}} - U^{2/(q-2)}\right\Vert_{q/2}^2\right) \ .\nonumber
 \end{eqnarray}
Finally, we use the fact that $\|\phi\|_q$ is a constant depending on $q$ and $d$, but not on $\phi\in\mathcal G$. (Indeed, by a virial-type theorem $\|\phi\|_q$ can be expressed in terms of $\mathcal C_{\gamma,d}$.) This proves the stability bound for $2<q<4$. The case $q\geq 4$ is similar. 
\end{proof}

To summarize the content of this section, we have reduced the proof of Theorem \ref{main} to the proofs of Theorems \ref{remainder} and \ref{becor}. Those will be given in Sections \ref{sec:holder} and \ref{sec:gns}, respectively.

%%%%%%%%%%%%%%%%%%%%%%%%%%%%%%%%%%%%%%%%%%%%%%%%%%%%%%%%%%%%%%%%%%%%%%%%%%%%%%%%%%%%%%%%%%%

\section{Remainder in H\"older's inequality and other consequences of uniform convexity}\label{sec:holder}

\subsection{Stability for H\"older's inequality}

In this section,
in contrast to the rest of this paper,
we  work with \emph{complex-valued} functions on a general measure space  $(X,{\rm d}\mu)$.
Theorem \ref{holder} gives a stability estimate for H\"older's inequality which, to the best to our knowledge, is new. A recent theorem of Aldaz \cite{Al} gives  a lower bound on $\|f\|_p\|g\|_{p'}  - \int_X|f| |g|{\rm d}\mu$
in terms of the $L^2$ distance between $(|f|/\|f\|_p)^{p/2}$ and  $(|g|/\|g\|_{p'})^{p'/2}$. Our bound takes into account differences of the phase.
Though Aldaz's bound would have sufficed for many of our applications -- when we consider potentials of a single sign --  we discovered it only after the first version of this paper  was completed. The paper was brought to our attention by P. Sosoe to whom we are grateful.

First we define the {\it duality map} ${\mathcal
D}_p$ on functions 
from $L^p$ to the unit sphere in $L^{p'}$:
$${\mathcal D}_p(f) (x)=
\|f\|_p^{1-p} |f|^{p-2}(x)\overline{f(x)}
$$
The map has the property that  $\int_X {\mathcal D}_p(f) f{\rm d}\mu  =
\|f\|_p$, and ${\mathcal D}_p(f)$ is 
the unique unit vector in the dual space $L^{p'}(X,\mu)$ to $L^p(X,\mu)$
that has this property. Here and below, $p'$ is related to $p$ by $1/p+1/p' = 1$.

\begin{theorem}[H\"older's inequality with remainder] \label{holder}
Let $p\geq 2$. Let $f$ be a unit vector in $L^p(X,\mu)$, and let $g$ be a unit vector in $L^{p'}(X,\mu)$. Then 
we have both
\begin{equation}\label{main1}
\left|\int_X fg \,{\rm d}\mu\right| \leq 1 -  \frac{p'-1}{4}\|{\mathcal D}_p(f) -e^{i\theta}g\|_{p'}^2\ ,
\end{equation}
and  
\begin{equation}\label{main2}
\left|\int_X fg \,{\rm d}\mu\right| \leq 1 -  \frac{1}{p\ 2^{p-1}}\| e^{i\theta}f - {\mathcal D}_{p'}(g)\|_{p}^{p}\ .
\end{equation}
where  $\theta \in [0,2\pi)$ is such that  $e^{i\theta}\int_X fg  \,{\rm d}\mu$ is positive.
The exponents $2$ and $p$ on the right sides of \eqref{main1} and \eqref{main2} are best possible. 
\end{theorem}

\begin{remark}
When estimating a remainder in H\"older's inequality, there are always two
choices:  One can apply the appropriate duality map 
to either of the two functions. Generally, one choice will be better than the other in that  the exponent in (\ref{main1}) is always $2$, 
while the exponent $p$ in (\ref{main2}) is larger than $2$ except when $p=p'
=2$.  This will be illustrated in the proof of Theorem~\ref{remainder},
which we give first, before proving Theorem \ref{holder}.
\end{remark}

\begin{proof}[Proof of Theorem \ref{remainder}]
We recall that  the functional $\mathcal H[\psi,U]$, defined in (\ref{hdef}), is non-negative
due to H\"older's inequality for $L^{q/2}$ and $L^{q/(q-2)}$.   

When $q\geq 4$, $q/(q-2) \leq 2$, and we may apply (\ref{main1})
with $g := U$ and $f  := |\psi|^2/\|\psi\|_q^2$. Then since 
${\mathcal D}_{q/2}(f) = |\psi|^{q-2}/\|\psi\|_q^{q-2}$,  we obtain (\ref{hlow1}).

When $2\leq q \leq 4$, $1 \leq q/2 \leq 2$, and we may apply (\ref{main1}) with $f := U$ and $g := 
|\psi|^2/\|\psi\|_q^2$. Then since 
${\mathcal D}_{q/(q-2)}(f) = U^{2/(q-2)}$,  we obtain (\ref{hlow2}).
\end{proof}

\begin{proof}[Proof of Theorem \ref{holder}]
Recall that the $L^{p}$ spaces, $1 < p < \infty$, are uniformly convex, and specifically, for $1< p \leq 2$ one has that for any
two unit vectors 
$u,v\in L^p(X,{\rm d}\mu)$, 
\begin{equation}\label{step2}
\left\Vert \frac{u+v}{2}\right\Vert_p  \leq 1 - \frac{p-1}{8}\|u-v\|_p^2\ ,
\end{equation}
which is the expression of $2$-uniform convexity of the $L^p$ norms, $1 < p \leq 2$. For the constant, see \cite{BCL}.

Similarly,  as a direct consequence of the `easy Clarkson inequality'
\cite[eq. (2.2) and after (2.4)]{BCL}, for
$2\leq p < \infty$, one has
for any
two unit vectors 
$u,v\in L^p(X,{\rm d}\mu)$, 
\begin{equation}\label{step2A}
\left\Vert \frac{u+v}{2}\right\Vert_p  \leq 1 - \frac{1}{p2^{p}}\|u-v\|_p^p\ ,
\end{equation}
and again \cite{BCL} may be consulted for proofs and references.

Now take $f$ to be a unit vector in $L^p$ and $g$ a unit vector in $L^{p'}$.   Let $\theta$ be such that $e^{i\theta} \int_X fg \,{\rm d}\mu$ is positive. 
By H\" older's inequality, 
\begin{equation}\label{step1}
1 + e^{i\theta}\int_X fg \,{\rm d}\mu = \int_X f({\mathcal D}_p(f) + e^{i\theta}g) \,{\rm d}\mu \leq \|{\mathcal D}_p(f)+ e^{i\theta}g\|_{p'}\ .
\end{equation}
Since both ${\mathcal D}_p(f)$ and $e^{i\theta}g$ are unit vectors in $L^{p'}(X,{\rm d}\mu)$, when   $2\leq p < \infty$, we may combine
 (\ref{step2}) and (\ref{step1}) to obtain
$$
\frac12 +  \frac12 \left| \int_X fg \,{\rm d}\mu  \right| \leq 1 - \frac{p'-1}{8}\|{\mathcal D}_p(f) -e^{i\theta}g\|_{p'}^2\ ,
$$
which reduces to (\ref{main1}).
Likewise, for $1< p \leq 2$, by (\ref{step2A}) and (\ref{step1}),
$$
\frac12 +  \frac12\left|\int_X fg \,{\rm d}\mu \right|  \leq 1 - \frac{1}{p'2^{p'}}\|{\mathcal D}_p(f) -g\|_{p'}^{p'}\ ,
$$
 which, upon interchanging $f$ with $g$, and $p$ with $p'$, reduces to (\ref{main2}). 

It remains to show that the exponents in these inequalities are sharp. This is evident for (\ref{main1}) since in a finite dimension setting where the norms are twice continuously differentiable, it is evident that the exponent $2$ is best possible. Hence we focus on (\ref{main2}).

Let $X= [0,1]$ and let $\mu$ be Lebesgue measure. 
For a small $\delta>0$, define
$$
f := 1 \qquad{\rm and}\qquad 
g := \begin{cases} (1-\delta)^{-1/p'} & 0 \leq x \leq 1-\delta\\
0 &1-\delta < x \leq 1\end{cases} \ .
$$
Then $f$ and $g$ are unit vectors in $L^p(X,\mu)$ and $L^{p'}(X,\mu)$ respectively.
$$
\int_X fg \,{\rm d}\mu = (1-\delta)^{1/p} =  1 - \frac{1}{p}\delta + {\mathcal O}(\delta^2)\ ,
$$
while a simple computation yields
$$
\|f- {\mathcal D}_{p'}(g)\|_p = \left(\left(\left(1-\delta\right)^{-1/p}-1\right)^p \left(1-\delta\right) +\delta\right)^{1/p} > \delta^{1/p}\ .
$$
Hence, for this choice of $f$ and $g$,
${\displaystyle 
 \int_X fg\,{\rm d}\mu \geq 1 - \frac{1}{p}\|f - {\mathcal D}_{p'}(g)\|_p^p}$
 for sufficiently small $\delta>0$.
\end{proof}

%%%%%%%%%%%%%%%%%%%%%%%%%%%%%%%%%%%%%%%%%%%%%%%%

\subsection{Further consequences of uniform convexity}

The uniform convexity inequalities have one more consequence that is useful for us, since \begin{equation}\label{double}
\frac{|f|^{q-2}}{\|f\|_q^{q-2}} -  \frac{|g|^{q-2}}{\|g\|_q^{q-2}} =
{\mathcal D}_{q/2}\left( \frac{|f|^2}{\|f\|_q^2}\right) - 
{\mathcal D}_{q/2}\left(\frac{|g|^2}{\|g\|_q^2}\right)\ .
\end{equation}
The  quantity on the left side appears in (\ref{rem1}), our stability bound
for the GNS inequality. 
The next lemma  shows that for $q/2\geq 2$, the duality map is Lipschitz continuous, which allows us
estimate the $L^{q/(q-2)}$ norm of the left side of (\ref{double}) in terms
of $\|f - g\|_q$; see Lemma~\ref{lip} below. 
Then in Section~\ref{sec:gns}, we shall prove an $L^q$ stability theorem for the GNS inequality, which then
leads to the more specialized bound in Theorem~\ref{becor}.

\begin{lemma}[H\"older continuity of the duality map] \label{modcont}
Let $f,g \in L^p(X,\mu)$. Then, for $p \geq 2$, 
\begin{equation}\label{ng2}
\|{\mathcal D}_p(f) - {\mathcal D}_p(g)\|_{p'} \leq 4(p-1)\ \frac{\|f-g\|_p}{\|f\|_p+\|g\|_p} \ ,
\end{equation}
and for $1 < p \leq 2$, 
\begin{equation}\label{ng1}
\|{\mathcal D}_p(f) - {\mathcal D}_p(g)\|_{p'} \leq 2\left(p'\frac{\|f-g\|_p}{\|f\|_p+\|g\|_p}\right)^{p-1} \ .
\end{equation}
\end{lemma}

\begin{proof} By H\"older's inequality,
\begin{eqnarray}
\|{\mathcal D}_p(f) + {\mathcal D}_p(g)\|_{p'} \left(\|f\|_p+\|g\|_p \right)
&\geq& \|{\mathcal D}_p(f) + {\mathcal D}_p(g)\|_{p'} \|f+g\|_p \nonumber\\
&\geq& \int_X ({\mathcal D}_p(f) + {\mathcal D}_p(g))(f+g) \,{\rm d}\mu \nonumber\\
&=& 2(\|f\|_p + \|g\|_p) - \int_X ({\mathcal D}_p(f) - {\mathcal D}_p(g))(f-g)\,{\rm d}\mu \nonumber\\
&\ge&  2 (\|f\|_p + \|g\|_p) -  \|{\mathcal D}_p(f) - {\mathcal D}_p(g)\|_{p'} \|f-g\|_p \ .\nonumber
\end{eqnarray}
Thus,
\begin{equation*}
1 - \left\Vert \frac{  {\mathcal D}_p(f) + {\mathcal D}_p(g)}{2}  \right\Vert_{p'} \leq  
 \left\Vert \frac{{\mathcal D}_p(f) - {\mathcal D}_p(g)}{2}\right\Vert_{p'} \, \frac{\|f-g\|_p}{\|f\|_p+\|g\|_p} \ .
 \end{equation*}
Now apply the uniform convexity inequalities to the left side and simplify.
\end{proof} 

\begin{lemma}\label{lip}
Let $f,g\in L^q(X,\mu)$   Then for all $q\geq 2$, 
\begin{equation}\label{qo2est}
\max\{\, \|f\|_q\ ,\ \|g\|_q\} \left\Vert \frac{|f|^2}{\|f\|_q^2} - \frac{|g|^2}{\|g\|_q^2}\right\|_{q/2} 
\leq  4 \|f-g\|_q\ .
\end{equation}
Moreover, if $q\geq 4$, then
\begin{equation}\label{bqest}
\max\{\, \|f\|_q\ ,\ \|g\|_q\} \left\| \frac{|f|^{q-2}}{\|f\|_q^{q-2}} - \frac{|g|^{q-2}}{\|g\|_q^{q-2}} \right\|_{q/(q-2)}
\leq 4(q-2)\|f-g\|_q\,.
\end{equation}
\end{lemma}

\begin{proof}
Without loss of generality, we may assume that $\|f\|_q \geq \|g\|_q$. Note that
$$
\frac{|f|^2}{\|f\|_q^2} - \frac{|g|^2}{\|g\|_q^2} = \frac{|f|^2- |g|^2 }{\|f\|_q^2} + \frac{|g|^2}{\|f\|_q^2\|g\|_q^2} \left(\|g\|_q^2 - \|f\|_q^2\right)\ .
$$
Taking the $L^{q/2}$ norm of both sides, and using the triangle inequality, we obtain
$$
\left\Vert \frac{|f|^2}{\|f\|_q^2} - \frac{|g|^2}{\|g\|_q^2}\right\Vert_{q/2} 
\leq  \frac{2}{\|f\|_q^2} \left\||f|^2-|g|^2\right\|_{q/2} 
= \frac{2}{ \max\{\, \|f\|_q^2\ ,\ \|g\|_q^2\}} \left\||f|^2-|g|^2\right\|_{q/2} .
$$
Then, by the Cauchy-Schwarz inequality, 
$$
\left\||f|^2-|g|^2\right\|_{q/2} \leq  \left\||f|-|g|\right\|_q \left\||f|+|g|\right\|_q 
\leq   2\left\|f-g\right\|_q \max\left\{\, \|f\|_q\ ,\ \|g\|_q\right\}\ .
$$
Combining the last two estimates yields (\ref{qo2est}).

Next, for $q\geq 4$, apply (\ref{ng2})   with $p =q/2$ and with $|f|^2/\|f\|_q^2$ in place of $f$, and $|g|^2/\|g\|_q^2$ in place of $g$ to get
$$
\left\| \frac{|f|^{q-2}}{\|f\|_q^{q-2}} - \frac{|g|^{q-2}}{\|g\|_q^{q-2}} \right\|_{q/(q-2)} 
\leq (q-2)\left\Vert \frac{f^2}{\|f\|_2^2} -   \frac{g^2}{\|g\|_2^2}\right\|_{q/2}\,.
$$
Combining this with (\ref{qo2est}) yields (\ref{bqest})
\end{proof}

%%%%%%%%%%%%%%%%%%%%%%%%%%%%%%%%%%%%%%%%%%%%%%%%%%%%%%%%%%%%%%%%%%%%%%%%%%%%%%%%%%%%%%%%%%%%

\section{A GNS-type inequality}\label{sec:gns0}

\subsection{The GNS minimization problem}

In this section we consider the minimization problem
\begin{equation}
 \label{eq:inf2}
-\mathcal C'_{q,d} = \inf\left\{ \int_{\R^d} |\nabla\psi|^2 \,dx - \left( \int_{\R^d} |\psi|^q \,dx \right)^{2/q} :\ \psi\in H^1(\R^d)\,,\ \int_{\R^d} \psi^2\,dx =1 \right\} \,.
\end{equation}
(Here, and everywhere in the following two sections we only consider \emph{real-valued} functions $\psi$.)

Up to now, $q$ has been tied to $\gamma$ by \eqref{eq:dual}, but in the following two sections $\gamma$ does not appear anymore and $q$ is the natural variable. Its restrictions, which we will assume throughout, are $2<q<\infty$ if $d=1,2$ and $2<q<2d/(d-2)$ if $d\geq 3$. (This corresponds to the conditions on $\gamma$ from Theorems \ref{unique} and \ref{main}.) Since $q$ and not $\gamma$ is the natural variable, we write $\mathcal C'_{q,d}$ instead of $\mathcal C_{\gamma,d}$ in the following two sections. These two numbers coincide if $\gamma$ and $q$ are related by \eqref{eq:dual}.

We shall prove that the infimum in \eqref{eq:inf2} is finite and is attained and that the minimizer is unique (up to translations and a sign). We record these facts in the following theorem. Recall that $\mathcal G$ denotes the set of optimizers in \eqref{eq:inf2}, see \eqref{eq:g}.

\begin{theorem}\label{kwong}
Let $2<q<\infty$ if $d=1,2$ and $2<q<2d/(d-2)$ if $d\geq 3$. There is a radial, strictly decreasing function $Q$ such that
$$
\mathcal G = \left\{ \sigma Q(\cdot -a): \ a\in\R^d\,,\ \sigma=\pm1 \right\} \,.
$$
The function $Q$ satisfies $\int Q^2\,dx = 1$ and
\begin{equation}
\label{eq:eleq}
-\Delta Q - \|Q\|_{q}^{2-q} Q^{q-1} = E Q \,,
\end{equation}
where $E=-\mathcal C_{q,d}'$.
\end{theorem}

This theorem is (essentially) known. We include a proof for the sake of completeness and in order to set up the notation for the proof of Theorem \ref{main}. The fact that $\mathcal C'_{\gamma,d}$ is finite and is attained is a standard fact from the calculus of variations. While a short proof can be based on the method of symmetric decreasing rearrangement, we will employ the compactness theorem of \cite{Li} which also yields the fact that minimizing sequence converge strongly in $H^1$ (up to translations). This fact will be useful in our proof of Theorem \ref{main}. The deepest part of this theorem, namely the uniqueness for $d\geq 2$, is a celebrated result of Kwong \cite{K}, who extended works in \cite{C,MS}. Note that for $d\geq 2$, the optimizing function $Q$ is not explicitly known. In contrast, for $d=1$ the optimizers where explicitly found by Keller \cite{Ke} by solving the corresponding Euler--Lagrange equation (although Keller did not prove the existence of a minimizer nor discuss the regularity needed for the solution of the Euler--Lagrange equation). The $d=1$ result is also implicitly contained in \cite{N}. 

%%%%%%%%%%%%%%%%%%%%%%%%%%%%%%%%%%%%%%%%%%%%%%

\subsection{Existence of an optimizer}

We begin by showing that the infimum \eqref{eq:inf2} is finite. This variational problem may be put in a more familiar form by replacing $\psi(x)$ with $\lambda^{d/2}\psi(\lambda x)$ and optimizing over $\lambda$. This leads to
\begin{equation}\label{GNSrel}
\mathcal C_{q,d}'=\theta^{1/(1-\theta)}(1-\theta) \mathcal S_{q,d}^{-1/(1-\theta)} \,,
\end{equation}
where
$$
\mathcal S_{q,d} = \inf_{\psi\in H^1(\R^d)}
\frac{ \left(\int_{\R^d} |\nabla\psi|^2\,dx \right)^\theta \left(\int_{\R^d} \psi^2\,dx \right)^{1-\theta}}{\left(\int_{\R^d} |\psi|^q \,dx \right)^{2/q} } \,, \qquad (d-2)\theta + d(1-\theta)= 2d/q \,.
$$
Thus, the determination of $\mathcal C'_{q,d}$ comes down to the determination of the best constant  $\mathcal S_{q,d}$ in the  Gagliardo--Nirenberg--Sobolev (GNS) inequality
\begin{equation}\label{gns}
\left(\int_{\R^d} |\nabla\psi|^2\,dx \right)^\theta \left(\int_{\R^d} \psi^2\,dx \right)^{1-\theta}
\geq \mathcal S_{q,d}  \left(\int_{\R^d} |\psi|^q \,dx \right)^{2/q} \,,
\end{equation}
which is known to be strictly positive. This shows that $\mathcal C_{q,d}'$ is finite.

Next, we prove existence of a minimizer and compactness of minimizing sequences.

\begin{lemma}\label{befar}
The infimum \eqref{eq:inf2} is attained. Moreover, if $(\psi_n)\subset H^1(\R^d)$ is a minimizing sequence for $-\mathcal C_{q,d}'$, then
$$
\lim_{n\to\infty} \inf_{\phi\in\mathcal G} \| \psi_n-\phi\|_{H^1} = 0 \,.
$$
\end{lemma}

\begin{proof}
 It suffices to show that any minimizing sequence has a subsequence which (up to translations) converges strongly in $H^1$ to a minimizer. Thus, let $(\psi_n)\subset H^1(\R^d)$ be such that $\|\psi_n\|=1$ and
$$
\mathcal E[\psi_n] = \int_{\R^d} |\nabla\psi_n|^2 \,dx - \left( \int_{\R^d} |\psi_n|^q \,dx \right)^{2/q} \to -\mathcal C_{q,d}' \,.
$$
It follows from \eqref{gns} that $\|\nabla \psi_n\|_2$ is bounded from above uniformly in $n$. Since $-C_{q,d}'<0$ (as a trial function argument shows), we know that $\|\psi_n\|_q$ is bounded away from zero uniformly in $n$. The uniform boundedness of $\|\nabla\psi_n\|_2$ and $\|\psi_n\|_2$, together with \eqref{gns}, implies that there is an $r>q$ (satisfying $r<2d/(d-2)$ if $d\geq 3$) such that $\|\psi_n\|_r$ is bounded from above uniformly in $n$. By the $pqr$-theorem \cite{FrLiLo} (see also \cite[Ex. 2.22]{LiLo}) we conclude that there are $\delta>0$ and $\epsilon>0$ such that for all $n$,
$$
|\{ |\psi_n|> \epsilon \}| \geq \delta \,.
$$
This means that the sequence $(\psi_n)$ does not tend to zero in measure. Therefore, the assumptions of the Compactness up to Translations Theorem of Lieb \cite{Li} (see also \cite[Thm. 8.10]{LiLo}) are satisfied and we obtain a sequence of vectors $(y_n)\subset\R^d$ such that, after passing to a subsequence if necessary, $\psi_n(\cdot + y_n)$ has a weak limit $\psi$ in $H^1(\R^d)$ with $\psi\not\equiv 0$. By the Rellich--Kondrachov theorem, passing to another subsequence if necessary, we can also assume that $\psi_n(\cdot+y_n)$ converges to $\psi$ pointwise almost everywhere.

We now argue that $\psi_n(\cdot+y_n)$ converges \emph{strongly} in $H^1$ and that $\psi$ is a minimizer. The weak convergence in $H^1$ implies that
$$
1= \|\psi_n\|_2^2 = \|\psi\|_2^2 + \|\psi_n-\psi\|_2^2 + o(1)
$$
and
$$
\|\nabla\psi_n\|_2^2 = \|\nabla\psi\|_2^2 + \|\nabla(\psi_n-\psi)\|_2^2 + o(1) \,.
$$
The almost everywhere convergence together with the Br\'ezis--Lieb lemma \cite{BL1} (see also \cite[Thm. 1.9]{LiLo}) implies that
$$
\|\psi_n\|_q^q = \|\psi\|_q^q + \|\psi_n-\psi\|_q^q + o(1) \,.
$$
Thus, by concavity of $x\mapsto x^{2/q}$,
\begin{equation*}
\limsup_{n\to\infty} \|\psi_n\|_q^2 \leq \|\psi\|_q^2 + \limsup_{n\to\infty} \|\psi_n-\psi\|_q^2 \,.
\end{equation*}
We conclude that
$$
-\mathcal C_{q,d}' = \lim_{n\to\infty} \mathcal E[\psi_n] \geq \mathcal E[\psi] + \liminf_{n\to\infty} \mathcal E[\psi_n-\psi] \,.
$$
On the right side, we bound
\begin{equation*}
\mathcal E[\psi_n-\psi] \geq -\mathcal C_{q,d}' \|\psi_n-\psi\|_2^2  = -\mathcal C_{q,d}' \left(1-\|\psi\|_2^2 +o(1) \right)
\end{equation*}
and, after adding $\mathcal C_{q,d}' \left(1-\|\psi\|_2^2 \right)$ to both sides, we infer that
$$
-\mathcal C_{q,d}' \|\psi\|_2^2 \geq \mathcal E[\psi] \,.
$$
This means that $\psi/\|\psi\|_2$ is an optimizer.

We also conclude that all the inequalities so far must have been asymptotically equalities. That is,
\begin{equation}
\label{eq:strongconv}
\lim_{n\to\infty} \mathcal E[\psi_n-\psi] = -\mathcal C_{q,d}' \left(1-\|\psi\|_2^2 \right)
\end{equation}
and
$$
\limsup_{n\to\infty} \|\psi_n\|_q^2 = \|\psi\|_q^2 + \limsup_{n\to\infty} \|\psi_n-\psi\|_q^2 \,.
$$
From the second equation, together with the \emph{strict} concavity of $x\mapsto x^{2/q}$, we conclude that $\limsup_{n\to\infty} \|\psi_n-\psi\|_q^2 = 0$. Thus,
$$
\liminf_{n\to\infty} \mathcal E[\psi_n-\psi] = \liminf_{n\to\infty} \|\nabla (\psi_n-\psi)\|_2^2 \geq 0 \,.
$$
Comparing this with \eqref{eq:strongconv} we infer that $\|\psi\|_2=1$ and $\lim_{n\to\infty} \|\nabla (\psi_n-\psi)\|_2 = 0$. This proves that $\psi_n(\cdot+y_n)$ converges strongly in $H^1$ to $\psi$, which is an optimizer. The proof of Lemma \ref{befar} is complete.
\end{proof}

%%%%%%%%%%%%%%%%%%%%%%%%%%%%%%%%%%%%%%%%%%%%%%%%%%%%%%%%%%%%%%%%%%%%%%

\subsection{Uniqueness of the optimizer}

We now complete the proof of Theorem \ref{kwong} by showing how to deduce the uniqueness of an optimizer from Kwong's theorem.

\begin{proof}[Proof of Theorem \ref{kwong}]
Let $Q$ be a minimizer for $-\mathcal C'_{q,d}$ (whose existence we have shown in the proof of Lemma \ref{befar}). We argue that $Q$ does not change sign. Indeed, both $Q_+$ and $Q_-$ belong to $H^1$ and
\begin{align*}
-\mathcal C'_{q,d} = \|\nabla Q \|_2^2 - \|Q\|_q^2 & = \|\nabla Q_+\|_2^2 + \|\nabla Q_-\|_2^2 - \left( \|Q_+\|_q^q +\|Q_-\|_q^q \right)^{2/q} \\
& \geq \|\nabla Q_+\|_2^2 - \|Q_+\|_q^2 + \|\nabla Q_-\|_2^2 - \|Q_-\|_q^2 \\
& \geq -\mathcal C'_{q,d} \|Q_+\|_2^2 -\mathcal C'_{q,d} \|Q_-\|_2^2 = -\mathcal C'_{q,d} \,,
\end{align*}
where we used the concavity of the map $x\mapsto x^{2/q}$ (recall that $q>2$). Since this map is even \emph{strictly} concave, we obtain
$\left( \|Q_+\|_q^q +\|Q_-\|_q^q \right)^{2/q}< \|Q_+\|_q^2 + \|Q_-\|_q^2$ unless one of $Q_+$ and $Q_-$ is identically zero. This strict inequality, however, would contradict the fact that $Q$ is a minimizer and, therefore, we conclude that one of $Q_+$ and $Q_-$ is, indeed, identically zero, that is, $Q$ does not change sign.

Thus, up to multiplying $Q$ by $-1$ we may assume that $Q$ is non-negative. A straightforward computation shows that $Q$ satisfies the Euler--Lagrange equation \eqref{eq:eleq} in the weak sense. Standard regularity results (see, e.g., \cite[Thm. 11.7]{LiLo}) show that $Q$ is a $C^\infty$ function. Moreover, $Q$ is ground state of the Schr\"odinger operator $-\Delta -\|Q\|_q^{2-q} Q^{q-1}$ in $L^2(\R^d)$ and, therefore, it and its derivatives decay exponentially. By the method of moving planes \cite{GNN} one sees that, after a translation if necessary, $Q$ is a radial function and is strictly decreasing with respect to the distance from the origin.

We now prove the \emph{uniqueness} of minimizers for $\mathcal C'_{q,d}$ up to translations. Let $Q$ and $P$ be minimizers. Then, by the above arguments, after a sign change and a translation if necessary, we may assume that $Q$ and $P$ are both symmetric decreasing functions centered around the same origin. It follows from \eqref{eq:eleq} (where the Lagrange multiplier is simply the value $E=-\mathcal C'_{q,d}$) that the functions $\tilde Q(x)=|E|^{-1/(q-2)} Q(x/\sqrt{|E|})/\|Q\|_q$ and $\tilde P(x)=P(x/\sqrt{|E|})/\|P\|_q$ satisfy
$$
-\Delta\tilde Q -\tilde Q^{q-1}= - \tilde Q
\quad\text{and}\quad
-\Delta\tilde P -\tilde P^{q-1}= - \tilde P \,.
$$
Now the uniqueness result of Kwong \cite{K} for $d\geq 2$ (and a straightforward explicit solution for $d=1$) implies that $\tilde Q\equiv\tilde P$, that is, $Q/\|Q\|_q\equiv P/\|P\|_q$. Since both $Q$ and $P$ are normalized in $L^2(\R^d)$ we conclude that $Q\equiv P$, as claimed. This completes the proof of Theorem \ref{kwong}.
\end{proof}

%%%%%%%%%%%%%%%%%%%%%%%%%%%%%%%%%%%%%%%%%%%%%%%%%%%%%%%%%%%%%%%%%%%%%%%%%%%%%%%%%%%%%%%%%%%%

\section{Remainder in a GNS-type inequality}\label{sec:gns}

\subsection{Stability for the GNS minimization problem}

We continue our investigation of the minimization problem \eqref{eq:inf2}. After having established in the previous section the existence and uniqueness (up to translations) of a minimizer, we now investigate the stability question. Our main result in this section is the following theorem.

\begin{theorem}[GNS inequality with remainder]\label{berem}
Let $2<q<\infty$ if $d=1,2$ and $2<q<2d/(d-2)$ if $d\geq 3$. Then there is a constant $c_{q,d}>0$ such that
$$
\int_{\R^d} |\nabla\psi|^2 \,dx - \left( \int_{\R^d} |\psi|^q \,dx \right)^{2/q} \geq -\mathcal C'_{q,d} + c_{q,d} \inf_{\phi\in\mathcal G} \| \psi-\phi\|_{H^1}^2
$$
for all $\psi\in H^1(\R^d)$ with $\int_{\R^d} \psi^2\,dx =1$.
\end{theorem}

\begin{remark} The  constant $c_{q,d}$ will be provided by a compactness argument and hence its value cannot be computed
or even estimated numerically, unlike $\mathcal C'_{q,d}$. 
Recent work of Dolbeault and Toscani \cite{DT} has has provided a stability bound for certain GNS inequalities with explicit constant, but
a quartic power where one would expect a quadratic power. However, the GNS inequalities they treat are not the ones that we  require here.
It remains an open problem to prove a version of Theorem~\ref{berem} with a computable constant $c_{q,d}$. For other results  on stability
for GNS inequalities and their application to non-linear evolution equations, see \cite{CF}.
\end{remark} 

This theorem, together with Lemma~\ref{lip}, readily yields Theorem~\ref{becor}, as we explain next, before turning to the proof of Theorem~\ref{berem} itself. Since Theorem~\ref{becor} was all that remained to be proven in order to prove our main result,  Theorem~\ref{main}, this will complete our work.

\begin{proof}[Proof of Theorem~~\ref{becor}]  Let $2<q<\infty$ if $d=1,2$ and $2<q<2d/(d-2)$ if $d\geq 3$. 
By the Sobolev Embedding Theorem (see \eqref{gns}),  there is a constant $s_{q,d}$
such that $s_{q,d} \|\phi-\psi\|_q \leq \| \psi-\phi\|_{H^1}$, and by the triangle inequality, $|\|\phi\|_q - \|\psi\|_q| \leq  \|\phi-\psi\|_q$. 
Hence for all   $\psi\in H^1(\R^d)$ with $\int_{\R^d} \psi^2\,dx =1$,
$$
\int_{\R^d} |\nabla\psi|^2 \,dx - \left( \int_{\R^d} |\psi|^q \,dx \right)^{2/q} \geq -\mathcal C'_{q,d} + \frac{c_{q,d}s_{q,d}}{2} \inf_{\phi\in\mathcal G} 
\left(\| \phi -\psi\|_q^2 + |\|\phi\|_q - \|\psi\|_q|^2\right) \,.
$$
Now using (\ref{qo2est}) to bound $\|\phi-\psi\|_q$ from below leads to (\ref{rem2}), and, for $q\geq 4$, using (\ref{bqest})
to bound $\|\phi-\psi\|_q$ from below leads to (\ref{rem1}).
\end{proof}

We now turn to the proof of Theorem \ref{berem}. The key step in the proof is the following `local version' of Theorem \ref{berem}, which establishes the desired inequality under the additional assumption that the distance from the set of optimizers is small. The precise statement reads as follows.

\begin{lemma}\label{beclose}
There are constants $\epsilon>0$ and $c_{q,d}>0$ such that
$$
\int_{\R^d} |\nabla\psi|^2 \,dx - \left( \int_{\R^d} |\psi|^q \,dx \right)^{2/q} \geq -\mathcal C_{q,d}' + c_{q,d} \inf_{\phi\in\mathcal G} \| \psi-\phi\|_{H^1}^2
$$
for all $\psi\in H^1(\R^d)$ with $\int_{\R^d} \psi^2\,dx =1$ and $\inf_{\phi\in\mathcal G} \| \psi-\phi\|_{H^1} \leq \epsilon$.
\end{lemma}

Assuming Lemma \ref{beclose} for the moment we complete the

\begin{proof}[Proof of Theorem \ref{berem}]
Because of Lemma \ref{beclose} it suffices to prove Theorem \ref{berem} for functions satisfying the additional assumption $\inf_{\phi\in\mathcal G} \| \psi-\phi\|_{H^1} > \epsilon$ for some $\epsilon>0$. We argue by contradiction and assume there is a sequence of functions $\psi_n\in H^1(\R^d)$ with $\int_{\R^d} \psi_n^2\,dx =1$, $\inf_{\phi\in\mathcal G} \| \psi_n-\phi\|_{H^1} > \epsilon$ and
$$
\int_{\R^d} |\nabla\psi_n|^2 \,dx - \left( \int_{\R^d} |\psi_n|^q \,dx \right)^{2/q} \leq -\mathcal C_{q,d}' + \delta_n \inf_{\phi\in\mathcal G} \| \psi_n-\phi\|_{H^1}^2
$$
where $\delta_n\to 0$. Since $\inf_{\phi\in\mathcal G} \| \psi_n-\phi\|_{H^1}$ is bounded from above uniformly in $n$, the sequence $(\psi_n)$ is a minimizing sequence for $-\mathcal C_{q,d}'$ and therefore, by Lemma \ref{befar}, $\inf_{\phi\in\mathcal G} \| \psi_n-\phi\|_{H^1} \to 0$ as $n\to\infty$. This is a contradiction.
\end{proof}

Thus, we have reduced the proof of Theorem \ref{berem} to the proof of Lemma \ref{beclose}, which we will prove in the remaining two subsections.

%%%%%%%%%%%%%%%%%%%%%%%%%%%%%%%%%%%%%%%%%%

\subsection{Non-degeneracy of the linearization}

The main ingredient in our proof of Lemma \ref{beclose} is the following theorem which can be deduced from Kwong's results for $d\geq 2$ (and well-known results for $d=1$).

\begin{theorem}\label{kwong2}
Let $2<q<\infty$ if $d=1,2$ and $2<q<2d/(d-2)$ if $d\geq 3$. Let $Q$ and $E$ be the function and the number from Theorem \ref{kwong} and consider the self-adjoint operator
$$
H= -\Delta - (q-1)\|Q\|_{q}^{2-q} Q^{q-2} - E + (q-2) \|Q\|_q^{2-2q} |Q^{q-1}\rangle\langle Q^{q-1}|
$$
in $L^2(\R^d)$. Then $H\geq 0$ and
$$
\ker H = \spa \{Q,\partial_1 Q,\ldots,\partial_d Q \} \,. 
$$
\end{theorem}

\begin{proof}[Proof of Theorem \ref{kwong2}]
We now discuss the kernel of the operator $H$. First, note that $H$ is the Hessian of the minimization problem for $E$, in the sense that,
$$
\left. \frac{d^2}{d\epsilon^2} \right|_{\epsilon=0} \left( \left\|\nabla \frac{Q+\epsilon\phi}{\|Q+\epsilon\phi\|_2} \right\|_2^2 - \left\| \frac{Q+\epsilon\phi}{\|Q+\epsilon\phi\|_2} \right\|_q^2 \right) = (\phi,H\phi) 
$$
for every $\phi\in H^1(\R^d)$. This is shown by the same arguments as in the proof of Lemma \ref{beclose}. Thus, since $Q$ is a minimizer, $H\geq 0$, proving the first claim.

The inclusion $\spa\{Q,\partial_1 Q,\ldots,\partial_d Q\}\subset\ker H$ is easy. Indeed, equation \eqref{eq:eleq} implies that $Q\in\ker H$ and, differentiating \eqref{eq:eleq}, we also infer that $\partial_i Q\in\ker H$ for any $i=1,\ldots,d$. (Note that $\partial_i Q\in H^1(\R^d)$ by the regularity and decay results for $Q$ mentioned above.)

Let us prove the opposite inclusion. Since $Q$ is a radial function, the operator $H$ commutes with rotation and, therefore, can be analysed separated in each angular momentum channel. (In dimension $d=1$, this means separately on even and odd functions.) Again, since $Q$ is radial, the functions $\partial_i Q$ correspond to angular momentum one. Indeed, they are of the form $f(r) x_j/r$, where $x_j/|x|$ is a spherical harmonic of degree one. Since $Q$ is symmetric decreasing, the function $f$ is negative and therefore, by the Perron--Frobenius theorem, the $\partial_i Q$'s are the ground states of $H$ in the channel of angular momentum one and there no further elements in the kernel of $H$ restricted to that subspace. Since the restriction of $H$ to the subspace of angular momentum $l$ contains an additional term $l(l+d-2)/r^2$, when comparing this operator with $l\geq 2$ to the operator $l=1$, we see that the restriction of $H$ to the subspace of angular momentum $l$ is strictly positive for $l\geq 2$ and has trivial kernel.

Thus, it remains to prove the operator $H$ when restricted to angular momentum $l=0$ (that is, to radial functions) has only a one-dimensional kernel spanned by $Q$. To prove this, let $\eta$ be a radial function in $\ker H$ with $(Q,\eta)=0$. Then
\begin{equation}
\label{eq:eta}
L\eta = \alpha Q^{q-1} \,,
\end{equation}
where $\alpha=-(q-2)\|Q\|_q^{2-2q} (Q^{q-1},\eta)$ and
$$
L= -\Delta -(q-1)\|Q\|_q^{2-q} Q^{q-2} -E \,.
$$
Because of the Euler--Lagrange equation \eqref{eq:eleq} satisfied by $Q$ we have
$$
L(\eta -\beta Q) = 0
$$
for $\beta = -(q-2)^{-1}\|Q\|_q^{q-2} \alpha$, that is, $\eta-\beta Q \in\ker L$. Now the non-degeneracy result of Kwong in $d\geq 2$ (and an explicit computation in $d=1$) implies that $\eta\equiv 0$, as claimed.

(Strictly speaking, Kwong's result concerns the operator
$$
\tilde L = -\Delta -(q-1) \tilde Q^{q-2} +1
$$
with $\tilde Q(x)= |E|^{-1/(q-2)} Q(x/\sqrt{|E|})/\|Q\|_q$ as before. But this operator is unitarily equivalent to the operator $|E|^{-1} L$ by scaling. In particular, the kernel of $\tilde L$ on radial functions is trivial if and only if the same is true for that of $L$.)
\end{proof}

%%%%%%%%%%%%%%%%%%%%%%%%%%%%%%%%%%%%%%%%%%

\subsection{Proof of Lemma \ref{beclose}}

With Theorem \ref{kwong2} at hand we can finally give the proof of Lemma \ref{beclose}.

\begin{proof}[Proof of Lemma \ref{beclose}]
Let $\psi\in H^1(\R^d)$ with $\int \psi^2\,dx =1$. After a translation and a change of sign, if necessary, we may assume that
$$
\inf_{\phi\in\mathcal G} \| \psi-\phi\|_{H^1} = \|\psi - Q\|_{H^1} \,.
$$
This implies that
\begin{equation}
\label{eq:ortho}
\left( \psi-Q,\partial_j Q \right)_{H^1} = 0 \qquad \text{for all}\ j=1,\ldots,d \,.
\end{equation}
Let us introduce $j=\psi-Q$ and note that, since $\int \psi^2\,dx =1 =\int Q^2\,dx$,
\begin{equation}
\label{eq:almostortho}
2 \int_{\R^d} Q j \,dx + \int_{\R^d} j^2 \,dx = 0 \,.
\end{equation}
We now make use of the fact that
$$
\left| |a+b|^q - a^q - q a^{q-1} b -\frac{q(q-1)}{2} a^{q-2} b^2 \right| \leq C \left( a^{q-2-\theta} |b|^{2+\theta} + |b|^q \right)
$$
for all $a>0$, $b\in\R$ and some $C$ (depending only on $q>2$) and $\theta=\min\{q-2,1\}$. Thus, by H\"older's inequality,
\begin{align}
\label{eq:expansionq}
\|\psi\|_q^q = \|Q\|_q^q + q \int_{\R^d} Q^{q-1} j \,dx + \frac{q(q-1)}{2} \int_{\R^d} Q^{q-2} j^2 \,dx + O( \|j\|_q^{2+\theta} + \|j\|_q^q )
\end{align}
with an implied constant depending only on $q$ and $d$ (through $\|Q\|_q$). Moreover, by H\"older and Sobolev inequalities
\begin{align*}
& \left| q \int_{\R^d} Q^{q-1} j \,dx + \frac{q(q-1)}{2} \int_{\R^d} Q^{q-2} j^2 \,dx + O( \|j\|_q^{2+\theta} + \|j\|_q^q ) \right| \\
& \qquad \leq \const \left( \|j\|_q + \|j\|_q^q \right)
 \leq \const \left( \|j\|_{H^1} + \|j\|_{H^1}^q \right)
\end{align*}
with constants depending only on $q$ and $d$. We conclude that there is an $\epsilon>0$ (depending only on $q$ and $d$) such that
$$
\left| \|\psi\|_q^q - \|Q\|_q^q \right| \leq \frac{1}{2} \|Q\|_q^q 
\qquad \text{provided}\ \|j\|_{H^1}\leq\epsilon \,.
$$
For such $j$ we can take the $2/q$-th power of \eqref{eq:expansionq} and obtain
\begin{align*}
\|\psi\|_q^2 = & \|Q\|_q^2 + 2 \|Q\|_q^{2-q} \int_{\R^d} Q^{q-1} j \,dx \\
& + (q-1) \|Q\|_q^{2-q} \int_{\R^d} Q^{q-2} j^2 \,dx - (q-2) \|Q\|_q^{2-2q} \left( \int_{\R^d} Q^{q-1} j\,dx \right)^2 \\
& + O( \|j\|_q^{2+\theta})
\end{align*}
with an implied constant depending only on $q$ and $d$. Recalling equation \eqref{eq:eleq} for $Q$ and condition \eqref{eq:almostortho} for $j$ we obtain
\begin{align}
\|\nabla\psi\|^2 - \|\psi\|_q^2 & = E + 2\int_{\R^d} \nabla Q\cdot\nabla j\,dx - 2 \|Q\|_q^{2-q} \int_{\R^d} Q^{q-1} j \,dx \notag \\
& \quad + \|\nabla j\|^2 - (q-1) \|Q\|_q^{2-q} \int_{\R^d} Q^{q-2} j^2 \,dx \notag \\
& \quad + (q-2) \|Q\|_q^{2-2q} \left( \int_{\R^d} Q^{q-1} j\,dx \right)^2 + O( \|j\|_q^{2+\theta}) \notag \\
& = E + (j,Hj) + O( \|j\|_q^{2+\theta}) \,. \label{eq:linearization}
\end{align}
We now define
$$
k = j - (Q,j)Q - \sum_{i=1}^d \frac{(\partial_i Q,j)}{\|\partial_i Q\|^2} \partial_i Q
$$
and note that, according to Theorem \ref{kwong}, $k$ is $L^2$-orthogonal to the kernel of $H$. Since the essential spectrum of $H$ starts at $-E>0$ there is a constant $g>0$ such that
$$
(j,Hj) = (k,Hk)\geq g\|k\|_2^2 \,.
$$
On the other hand, it is easy to see that there is a constant $C>0$ such that
$$
H \geq -\Delta - C \,.
$$
(Indeed, one can take $C= \left\| (q-1)\|Q\|_q^{2-q} Q^{q-2} + E \right\|_\infty = (q-1)\|Q\|_q^{2-q} Q(0)^{q-2} + E$.)
Thus, for every $0<\rho<1$,
$$
(j,Hj)=(k,Hk) \geq g(1-\rho) \|k\|_2^2 + \rho \|\nabla k\|_2^2 -\rho C \|k\|_2^2
$$
and, upon choosing $\rho=g/(g+C+1)$,
$$
(j,Hj) = (k,Hk) \geq \frac{g}{g+C+1} \|k\|_{H^1}^2 \,.
$$
Recalling the orthogonality conditions \eqref{eq:ortho} and \eqref{eq:almostortho} we compute
\begin{align*}
\|k\|_{H^1}^2 & = \|j\|_{H^1}^2 + |(Q,j)|^2 \|Q\|_{H^1}^2 + \sum_{i=1}^d \frac{|(\partial_i Q,j)|^2}{\|\partial_i Q\|^4} \|\partial_i Q \|_{H^1}^2 - 2 (Q,j) (j,Q)_{H^1} \\
& = \|j\|_{H^1}^2 + |(Q,j)|^2 \|Q\|_{H^1}^2 + \sum_{i=1}^d \frac{|(\partial_i Q,j)|^2}{\|\partial_i Q\|^4} \|\partial_i Q \|_{H^1}^2 + \|j\|_2^2 (j,Q)_{H^1} \,. 
\end{align*}
Here we used the fact that the $\partial_i Q$'s are $H^1$ orthogonal among each other and to $Q$. This simply follows from the fact that $\partial_i Q$ is a radial function times the spherical harmonic $x_i/|x|$ of degree one. Thus,
$$
(j,Hj) \geq \frac{g}{g+C+1} \|j \|_{H^1}^2 + O(\|j\|_{H^1}^3) \,.
$$
We insert this bound into \eqref{eq:linearization} and obtain, after decreasing $\epsilon$ if necessary,
$$
\|\nabla\psi\|^2 - \|\psi\|_q^2 \geq E + \frac{g}{2(g+C+1)} \|j \|_{H^1}^2 \,.
$$
This completes the proof of Lemma \ref{beclose}.
\end{proof}

%%%%%%%%%%%%%%%%%%%%%%%%%%%%%%%%%%%%%%%

%%%%%%%%%%%%%%%%%%%%%%%%%%%%%%%%%%%%%%%%%%%%%%%%%%%%

\bibliographystyle{amsalpha}

\end{document}